%% file: OnedimCD.tex
\pdfoutput=1
\documentclass[reqno, english, 12pt, final]{amsart}
\usepackage{xparse}

\NewDocumentCommand{\tens}{t_}
 {%
  \IfBooleanTF{#1}
   {\tensop}
   {\otimes}%
 }
\NewDocumentCommand{\tensop}{m}
 {%
  \mathbin{\mathop{\otimes}\displaylimits_{#1}}%
 }

\def\vint_#1{\mathchoice%
          {\mathop{\kern 0.2em\vrule width 0.6em height 0.69678ex depth -0.58065ex
                  \kern -0.8em \intop}\nolimits_{\kern -0.4em#1}}%
          {\mathop{\kern 0.1em\vrule width 0.5em height 0.69678ex depth -0.60387ex
                  \kern -0.6em \intop}\nolimits_{#1}}%
          {\mathop{\kern 0.1em\vrule width 0.5em height 0.69678ex
              depth -0.60387ex
                  \kern -0.6em \intop}\nolimits_{#1}}%
          {\mathop{\kern 0.1em\vrule width 0.5em height 0.69678ex depth -0.60387ex
                  \kern -0.6em \intop}\nolimits_{#1}}}
\def\vintslides_#1{\mathchoice%
          {\mathop{\kern 0.1em\vrule width 0.5em height 0.697ex depth -0.581ex
                  \kern -0.6em \intop}\nolimits_{\kern -0.4em#1}}%
          {\mathop{\kern 0.1em\vrule width 0.3em height 0.697ex depth -0.604ex
                  \kern -0.4em \intop}\nolimits_{#1}}%
          {\mathop{\kern 0.1em\vrule width 0.3em height 0.697ex depth -0.604ex
                  \kern -0.4em \intop}\nolimits_{#1}}%
          {\mathop{\kern 0.1em\vrule width 0.3em height 0.697ex depth -0.604ex
                  \kern -0.4em \intop}\nolimits_{#1}}}

\usepackage{a4wide}
\usepackage{mathtools}
\usepackage{amsthm}
\usepackage{amssymb}
\usepackage{hyperref}
\usepackage{color}
\usepackage[obeyFinal]{todonotes}
\mathtoolsset{showonlyrefs}
\usepackage{enumerate}
\usepackage[obeyFinal]{todonotes}

\usepackage{csquotes}

\newcommand{\R}{\mathbb{R}}
\newcommand{\N}{\mathbb{N}}

\renewcommand{\P}{\mathcal{P}}
\newcommand{\B}{\mathcal{B}}
\newcommand{\og}{\mathrm{OptGeo}}
\newcommand{\opt}[2]{\mathrm{Opt}(#1,#2)}
\newcommand{\id}{\mathrm{id}}
\newcommand{\m}{\mathfrak{m}}
\newcommand{\geo}{\mathrm{Geo}}
\newcommand{\Tan}{\mathrm{Tan}}
\newcommand{\spt}{\mathrm{spt}\,}
\newcommand{\restr}{\mathrm{restr}}
\renewcommand{\d}{\,\mathrm{d}}
\renewcommand{\liminf}{\varliminf}
\newcommand{\im}{\mathrm{Im}\,}

\newtheorem*{theorem*}{Theorem}
\newtheorem{theorem}{Theorem}[section]
\newtheorem{lemma}[theorem]{Lemma}
\newtheorem{proposition}[theorem]{Proposition}
\newtheorem{corollary}[theorem]{Corollary}
\theoremstyle{definition}
\newtheorem{definition}[theorem]{Definition}
\newtheorem{example}[theorem]{Example}
\newtheorem{question}{Question}

\theoremstyle{remark}
\newtheorem{remark}[theorem]{Remark}

\begin{document}
\title[On one-dimensionality of metric measure spaces]{On one-dimensionality of metric measure spaces}
\author{Timo Schultz}
\address{University of Jyvaskyla
\\ Department of Mathematics and Statistics
\\ P.O.Box 35
\\ FI-40014 University of Jyvaskyla}
\email{timo.m.schultz@jyu.fi}
\date{\today}

\begin{abstract} In this paper, we prove that a metric measure space which has at least one open set isometric to an interval, and for which the (possibly non-unique) optimal transport map exists from any absolutely continuous measure to an arbitrary measure, is a one-dimensional manifold (possibly with boundary). As an immediate corollary we obtain that if a metric measure space is a very strict $CD(K,N)$ -space or an essentially non-branching $MCP(K,N)$-space with some open set isometric to an interval, then it is a one-dimensional manifold. We also obtain the same conclusion for a metric measure space which has a point in which the Gromov-Hausdorff tangent is unique and isometric to the real line, and for which the optimal transport maps not only exist but are unique. Again, we obtain an analogous corollary in the setting of essentially non-branching $MCP(K,N)$-spaces.
\end{abstract}
\maketitle
\section{Introduction}
The strong interplay between optimal mass transportation and (metric) geometry has been acknowledged in the last few decades leading to a great number of applications for example in the study of geometric and analytic inequalities, in describing and defining curvature bounds, and in the regularity theory of partial differential equations. The optimal transport theory is useful both in generalising classical results from the theory of smooth manifolds to possibly singular spaces, and in obtaining new results even in the smooth setting. 

In the present paper we will use tools from optimal transport theory to obtain global topological/geometric information about a metric (measure) space from information near a single point in the space. More precisely, we will use the existence -- and in some cases uniqueness -- of an optimal transport map together with one-dimensionality at (Theorem \ref{micro1dgen} and Corollary \ref{microcor}), or near (Theorem \ref{mild1d}, Corollary \ref{cor1d} and Theorem \ref{MCPmild1d}), a point in the space to prove that the space in question is a one-dimensional manifold. Here, by one-dimensionality \emph{at} a point, we mean that the Gromov-Hausdorff tangent at that point is unique and isometric to the real line, and by one-dimensionality \emph{near} a point, we mean that the point has an open neighbourhood isometric to an open interval. 

Such a result was first proven in the setting of Ricci limit spaces in \cite{Honda} by Honda and generalised to the setting of $RCD^*(K,N)$-spaces in \cite{Kitabeppu-Lakzian} by Kitabeppu and Lakzian. In both papers, it is proven that the underlying space satisfying the synthetic Ricci curvature lower bound in question, is a one-dimensional manifold if it is one-dimensional at a single point. These two papers share a common (and natural) viewpoint coming from the structure theory: in both settings one could write the metric measure space up to a zero measure set as a union of sets $\mathcal{R}_k$ in which one has the existence and uniqueness of tangents isomorphic to $k$-dimensional Euclidean space \cite{Cheeger-Colding1997, Mondino-Naber}. In our setting such a decomposition of the space cannot be true in general, not least because of the allowance of Finslerian (type) structures. Note that it is still meaningful to ask what can be concluded from the existence of a one-dimensional part even if we would impose Finslerian type behaviour, since a priori the dimension of the space needs not be constant.

The study of the present paper was partially motivated by the results concerning the existence of optimal transport maps on the spaces having synthetic Ricci curvature bounded from below. A notion of Ricci curvature lower bound for (possibly singular) metric measure spaces -- the so-called $CD(K,N)$-condition (\emph{curvature dimension condition}) -- was introduced in the seminal works of Sturm \cite{Sturm,Sturm2} and of Lott and Villani \cite{Lott-Villani}. The existence and uniqueness of the optimal map in the setting of spaces with synthetic Ricci curvature lower bounds was first proven by Gigli in \cite{Gigli} for non-branching $CD(K,N)$-spaces, and then generalised to strong $CD(K,N)$-spaces by Rajala and Sturm in \cite{Rajala-Sturm} and by Gigli, Rajala and Sturm in \cite{Gigli-Rajala-Sturm}. In their paper, Rajala and Sturm introduced the notion called \emph{essential non-branchingness}, which turned out to be a useful generalisation of the non-branching assumption on metric measure spaces. In \cite{Cavalletti-Mondino}, Cavalletti and Mondino proved the existence and uniqueness of optimal transport maps in $MCP(K,N)$-spaces if one assumes that the underlying metric measure space is essentially non-branching. Then in \cite{Kell}, Kell generalised the result to spaces satisfying even weaker version of curvature lower bound, namely to the setting of \emph{qualitatively non-degenerate} spaces (studied by Cavalletti and Huesmann in the non-branching case in \cite{Cavalletti-Huesmann}) -- still under the essential non-branching assumption. 

Heuristically, the non-branching assumption prevents the geodesics of an optimal plan to intersect at intermediate times, while the curvature lower bound assumption forces them to intersect when the plan is assumed not to be induced by a map, hence the existence and uniqueness of optimal maps is obtained by combining these two. Therefore, while the uniqueness of the optimal map is lost if there exists an essential amount of branching geodesics,  one might still pursue the existence of such a map. This approach was taken in \cite{Schultz2018} (and continued in \cite{Schultz2019}), where the author proved the existence of optimal transport maps in the setting of so-called very strict $CD(K,N)$-spaces. We remark that while in general (branching) $MCP(K,N)$-space the existence of an optimal transport map might fail by the example in \cite{Ketterer-Rajala}, it is still not known whether optimal maps exist in general $CD(K,N)$-spaces. 

\subsection*{Acknowledgements} The author acknowledges the support by the Academy of Finland, project \#314789, and thanks the anonymous referee for carefully reading the paper.

\section{Preliminaries}
For the purposes of this paper, we will always assume that $(X,d,\m)$ is a metric measure space which is a complete, locally compact and separable length space $(X,d)$ equipped with a locally finite measure $\m$. We will also assume that $\spt\m=X$. The space of (constant speed, length minimising) geodesics parametrised by $[0,1]$ is denoted by $\geo(X)$, and it is equipped with the supremum distance.

\subsection{Optimal mass transportation}In this section, we introduce the basic notions of optimal transport theory which set the basis for the paper. In addition, we establish in Proposition \ref{ylim} a subtle detail about the existence and uniqueness of optimal transport maps in the case of non-geodesic spaces.

In the main results of the paper we are assuming the existence of an optimal transport map for the Monge--Kantorovich problem with quadratic cost. The reason for such a choice for the cost function lies in the connection between Wasserstein geodesics and optimal dynamical transport plans. Note that one could obtain similar results by considering cost functions of the form $d^p$ for $p\in(1,\infty)$ different from $2$, since the representation of Wasserstein geodesics by measures on the space of geodesics, and the corresponding existence results of transport maps remain true in this case.

The quadratic Monge--Kantorovich problem reads as follows. Let $\mu_0,\mu_1\in\P(X)$ be Borel probability measures on $X$. Consider the minimisation problem
\[W_2^2(\mu_0,\mu_1)\coloneqq\inf \int d^2(x,y)\d\sigma(x,y),\]
where the infimum is taken over all \emph{transport plans} $\sigma$, that is, over all Borel probability measures $\sigma\in\P(X\times X)$ with $\mu_0$ and $\mu_1$ as marginals ($\mathrm{P}^1_\#\sigma=\mu_0, \mathrm{P}^2_\#\sigma=\mu_1$). It is a standard fact in optimal transport theory that in the setting of complete and separable metric spaces, the above infimum is in fact a minimum. We will call a minimiser of the problem an \emph{optimal (transport) plan}, and denote the set of all optimal plans by $\opt{\mu_0}{\mu_1}$. 

The optimality of a plan can be characterised by the so-called $c$-cyclical monotonicity in the following way. Let $\sigma$ be a transport plan between measures $\mu_0$ and $\mu_1$ for which $W_2^2(\mu_0,\mu_1)<\infty$. Then $\sigma$ is optimal if and only if it is concentrated on a $c$-cyclically monotone set, that is, if there exists a set $\Gamma\subset{X\times X}$ so that $\sigma(\Gamma)=1$, and 
\[\sum_i d^2(x_i,y_i)\le \sum_i d^2(x_i,y_{\tau(i)})\]
for any finite set $\{(x_i,y_i)\}_i\subset \Gamma$ and for any permutation $\tau$.

The function $W_2(\cdot,\cdot)$ defines a metric on the subset $\P_2(X)\subset\P(X)$ of probability measures with finite second moment, that is $\mu\in\P(X)$ with $\int d^2(x,x_0)\d\mu(x)<\infty$ for some $x_0\in X$. The distance $W_2$ is the so-called \emph{Wasserstein distance} (or more precisely the 2-Wasserstein distance). Since $X$ is a complete,  separable and -- by Hopf--Rinow theorem -- geodesic space, so is the Wasserstein space $(\P_2(X),W_2)$. Moreover, a curve $(\mu_t)\subset \P_2(X)$ is a geodesic if, and only if, there exists a measure $\pi\in \P(\geo(X))$ so that $ \mu_t= (e_t)_\#\pi$ and $(e_0,e_1)_\#\pi\in\opt{\mu_0}{\mu_1}$ \cite{Lisini}. Here $e_t\colon \geo(X)\to \R$ is the evaluation map $\gamma\mapsto \gamma_t$. Such a $\pi$ is called an \emph{optimal (geodesic) plan}, and the set of all optimal geodesic plans is denoted by $\og(\mu_0,\mu_1)$.

We say that the optimal plan $\sigma\in\opt{\mu_0}{\mu_1}$ is induced by a map if there exists a Borel map $T\colon X\to X\times X$ such that $\sigma=T_\#\mu_0$ and $\mathrm{P}^1\circ T=\id$. Analogously, $\pi\in\og(\mu_0,\mu_1)$ is induced by a map $T\colon X\to \geo(X)$, if $\pi=T_\#\mu_0$ and $e_0\circ T=\id$.

In the setting of geodesic spaces, one can always lift the optimal plan $\sigma\in\opt{\mu_0}{\mu_1}$ to an optimal geodesic plan $\pi\in\og(\mu_0,\mu_1)$ by making a measurable selection of geodesics for each pair of points in $X$. Therefore, the question of existence of optimal maps in the level of plans in $\opt{\mu_0}{\mu_1}$ and of geodesic plans in $\og(\mu_0,\mu_1)$ are equivalent. Furthermore, if the optimal plan $\pi\in\og(\mu_0,\mu_1)$ is unique, so is $\sigma\in\opt{\mu_0}{\mu_1}$. Since the framework for this paper is that of complete and locally compact metric spaces, being a length space is equivalent to being geodesic. However, sometimes it is more natural to drop the assumption of local compactness and still impose conditions implying the length structure for the space (this is the case for example in the $CD(K,\infty)$-setting). We remark the following connection between existence of optimal maps for transport plans and for geodesic transport plans in the above-mentioned non-geodesic case, which may be of independent interest.

\begin{proposition}\label{ylim}
Let $(X,d,\m)$ be a metric measure space (possibly non-geodesic and non-locally-compact). Assume that for all $\mu_0,\mu_1\in\P_2^{ac}(X)$ the set $\og(\mu_0,\mu_1)$ is non-empty, and that each optimal dynamical plan $\pi\in\og(\mu_0,\mu_1)$ is induced by a map. Then, for any $\mu_0, \mu_1\in\P_2^{ac}(X)$, every optimal plan $\sigma\in\opt{\mu_0}{\mu_1}$ is induced by a map, granting also the uniqueness of the optimal plan.

In particular, every optimal plan $\sigma\in \opt{\mu_0}{\mu_1}$, $\mu_0,\mu_1\in\P_2^{ac}(X)$, can be lifted to a unique dynamical plan $\pi\in\og(\mu_0,\mu_1)$ for which $\sigma=(e_0,e_1)_\#\pi$.
\end{proposition} 
\begin{proof}
Let $\sigma\in\opt{\mu_0}{\mu_1}\subset\P(X^2)$. Suppose that $\sigma$ is not induced by a map. Then there exists a $\mu_0$-positive measure Borel set $A\subset X$ such that $\sigma_x$ is not a Dirac mass for any $x\in A$, where $\{\sigma_x\}_{x\in X}$ is a disintegration of $\sigma$ with respect to the projection $\mathrm{P}^1$. Write
\[A=\bigcup_{i,j\in\N}A_{ij},\]
where $A_{ij}\coloneqq \{x\in A: \sigma_x(B(\xi_i,1/j))\in(0,1)\}$, and $\{\xi_i\}_{i\in\N}$ is dense in $X$. Sets $A_{ij}$ are measurable, since the maps $x\mapsto \sigma_x(B(\xi_i,1/j))$ are measurable for all $i,j\in\N$ by the disintegration theorem. Since $\mu_0(A)>0$, there exist $i_0$ and $j_0$, such that $\mu_0(A_{i_0j_0})>0$. 

Define now $\sigma^1$ and $\sigma^2$ as
\[\sigma^1\coloneqq \sigma\lvert_{X\times B(\xi_{i_0},\frac1j_{0})}, \quad \sigma^2\coloneqq \sigma\lvert_{X\times(X\setminus B(\xi_{i_0},\frac1j_{0}))}.\]

For $k\in\{1,2\}$, define $\hat\sigma^k$ as the measure for which
\[\int f\d\hat\sigma^k\coloneqq \int f\frac{\min\{\rho^1_0,\rho^2_0\}}{\rho^k_0}\circ \mathrm{P}^1\d\sigma^k,\]
for all positive Borel functions $f$, where $\rho^k_0$ is the density of $\mathrm{P}^1_\#\sigma^k$ with respect to the reference measure $\m$. Here we use convention $\frac{\min\{\rho^1_0,\rho^2_0\}}{\rho^k_0}=0$, when $\rho^k_0=0$. Since the function $\frac{\min\{\rho^1_0,\rho^2_0\}}{\rho^k_0}\in[0,1]$, we have that $\hat\sigma^k$ is a well-defined and finite measure. By the definition of $A_{{i_0}{j_0}}$, we have that $\rho^k_0\neq 0$ for $\mu_0$-almost every $x\in A_{{i_0}{j_0}}$. In particular, by the absolute continuity of $\mu_0$, there exists an $\m$-positive measure set where $\rho^k_0\neq 0$, and thus $\hat \sigma^{k}$ is non-trivial for $k\in\{1,2\}$. 

For $k\in\{1,2\}$ and $j\in\{0,1\}$, write $\hat\mu^k_j\coloneqq \mathrm{P}^{j+1}_\#\hat\sigma^k$. Then we have that $\hat\mu^1_0=\hat\mu^2_0$, $\hat\mu^1_1\perp\hat\mu^2_1$, and $\hat\sigma^k\in\opt{\hat\mu^k_0}{\hat\mu^k_1}$ (or more precisely, $\frac1N\hat\sigma^k\in\opt{\frac1N\hat\mu^k_0}{\frac1N\hat\mu^k_1}$, where $N$ is the normalisation constant $N=\hat\mu_0^1(X)=\hat\mu_0^2(X)$). Since $\hat\sigma^k\ll \sigma$, we have that $\hat\mu^k_j\ll\m$. Thus by assumption, for $k\in\{1,2\}$, there exists an optimal dynamical plan $\pi^k\in\og(\hat\mu^k_0,\hat\mu^k_1)$. On the other hand, since $(\hat\sigma^1+\hat\sigma^2)\ll \sigma$ and since 
\[\int d^2(x,y)\d(\hat\sigma^1+\hat\sigma^2)(x,y)=\int d^2(x,y)\d((e_0+e_1)_\#(\pi^1+\pi^2))(x,y),\] we have that $\pi^1+\pi^2$ is an optimal dynamical plan between absolutely continuous measures $2\hat\mu_0^1$ and $\hat \mu_1^1+\hat\mu_1^2$ that is not induced by a map. Hence, we arrive to a contradiction with the assumption. 
\end{proof}
\begin{remark} In the above proof the full existence and uniqueness of optimal geodesic plans is not needed, but instead existence and uniqueness of optimal geodesic plans inside some linearly convex subset of $\P(\geo(X))$. In particular, the proof can be adapted to prove the uniqueness of the optimal plan in $\opt{\mu_0}{\mu_1}$ between absolutely continuous measures in the essentially non-branching $CD(K,\infty)$-spaces  by using the result \cite[Corollary 5.22]{Kell} of the existence and uniqueness of optimal geodesic plans among all plans $\pi$ for which $\mu_t=(e_t)_\#\pi$ is absolutely continuous for all $t\in[0,1]$. 
\end{remark}

\subsection{(Measured) Gromov--Hausdorff tangents}
There are different notions of blow-ups for metric (measure) spaces. The one that we will use is based on a convergence of pointed metric spaces in the Gromov--Hausdorff sense:
\begin{definition}[pointed Gromov--Hausdorff convergence]Let $(X_i)_{i\in\N}=(X_i,d_i,x_i)_{i\in\N}$ be a sequence of pointed, complete, separable and locally compact geodesic spaces. Then $X_i\longrightarrow X_\infty=(X_\infty,d_\infty,x_\infty)$, if for all $R>0$ there exists a sequence $\varepsilon_i\to 0$ and $(1,\varepsilon_i)$-quasi-isometries $f_i\colon \bar{B}(x_i,R)\to \bar{B}(x_\infty,R)$ with $x_i\mapsto x_\infty$.
\end{definition}
Here $f$ being $(1,\varepsilon)$-quasi-isometry is defined by requiring that $f$ is  1-biLipschitz up to an additive constant $\varepsilon>0$ with an $\varepsilon$-dense image.

Tangents of a metric space $X$ at a point $x\in X$ are then obtained by looking at sequences of the form $(X,\lambda_id,x)$ with $\lambda_i\to \infty$.
\begin{definition}Let $X$ be a complete, separable and locally compact geodesic space, and let $x\in X$. A pointed metric space $(Y,d_Y,y)$ is a Gromov--Hausdorff tangent of $X$ at $x$, $(Y,d_Y,y)\in \Tan(X,x)$, if there exists a sequence $\lambda_i\to\infty$ so that $(X,\lambda_id,x)\longrightarrow (Y,d_Y,y)$ in the pointed Gromov--Hausdorff topology. 
\end{definition}
For doubling metric spaces the set of tangents at each point is non-empty. On the other hand, in general tangents are not unique. 

We point out that there exist notions of tangents of metric measure spaces that take into account the convergence of the (normalised) measure, which are in many cases more suitable for the study of metric measure spaces. However, for our purposes it is enough to consider the convergence as a metric concept (keeping in mind that we are assuming $\spt\m=X$). 

\section{One-dimensionality of metric measure spaces}
In this section we provide a generalisation of the following theorem
\begin{theorem*}[{\cite[Theorem 3.7]{Kitabeppu-Lakzian}}] Let $(X,d,\m)$ be an $RCD^*(K,N)$ space for $K\in\R$ and $N\in(1,\infty)$. Assume that there exists a point $x_0\in X$ such that there exists a unique (up to an isomorphism) measured Gromov-Hausdorff tangent of $(X,d,\m)$ at $x_0$ isomorphic (as a pointed metric measure space) to $(\R,\lvert\cdot \lvert,c\mathcal{L}^1,0)$. Then for any $x\in X$, there exists a positive number $\varepsilon>0$ such that $B(x,\varepsilon)$ is isometric to $(-\varepsilon,\varepsilon)$ or to $[0,\varepsilon)$.
\end{theorem*}
In Theorem \ref{micro1dgen} we state the result implying one-dimensional manifold structure from assumptions on the one-dimensionality at a point at infinitesimal level, analogously to the original result, when imposing existence and uniqueness of optimal transport maps. In Theorem \ref{mild1d} we give a local counterpart of the result for the case where uniqueness of transport maps is lost. Proofs presented here take advantage of the existence of optimal transport maps (assumption which may be justified by the results in \cite{Gigli-Rajala-Sturm,Rajala-Sturm,Cavalletti-Mondino,Kell,Schultz2018}), and by that simplify the ones given for Theorem 3.7 (and hence for Theorem 1.1) in \cite{Kitabeppu-Lakzian}.
\begin{theorem}\label{mild1d}Let $(X,d,\m)$ be a metric measure space with the following properties:
\begin{enumerate}
\item \label{exmaps} For every $\mu_0\in\P_2^{ac}(X)$ and $\mu_1\in\P_2(X)$, there exists $\pi\in\og(\mu_0,\mu_1)$ that is induced by a map from $\mu_0$.
\item \label{liibalaiba} There exists a point $x\in X$, and a neighbourhood $B(x,r)$ isometric to an open interval in $\R$. 
\end{enumerate}
Then $X$ is a one-dimensional manifold, possibly with boundary. 
\end{theorem}
Due to the existence of optimal transport maps in very strict $CD(K,N)$ -spaces \cite{Schultz2019}, and in essentially non-branching $MCP(K,N)$-spaces \cite{Cavalletti-Mondino}, we get the following immediate corollary of Theorem \ref{mild1d}.
\begin{corollary}\label{cor1d}
Let $(X,d,\m)$ be a very strict $CD(K,N)$ -space ($N\in(1,\infty)$), or essentially non-branching $MCP(K,N)$-space (or ess.\ nb., qualitatively non-degenerate space, see \cite{Kell}). Suppose that there exists a point $x\in X$, and a neighbourhood $B(x,r)$ isometric to an open interval in $\R$. Then $X$ is a one-dimensional manifold, possibly with boundary.
\end{corollary} 
It is worth noticing, that while Corollary \ref{cor1d} is known to be false in general $MCP(K,N)$-space by the example given by Ketterer and Rajala in \cite{Ketterer-Rajala}, it remains still open in general $CD(K,N)$-space:
\begin{question}Let $(X,d,m)$ be a $CD(K,N)$-space, and $B(x,r)\subset X$ isometric to an open interval. Is $X$ a manifold (possibly with boundary)?
\end{question}

\begin{proof}[Proof of Theorem \ref{mild1d}] The beginning of the proof goes just as that of Theorem 3.7 in \cite{Kitabeppu-Lakzian}. Denote by $\mathcal{F}$ the set of all the points that have a neighbourhood isometric to an open interval. Clearly, $\mathcal{F}$ is an open set in $X$. Suppose now that $X\setminus \mathcal{F}\neq \varnothing$. Since $\mathcal{F}$ is assumed to be non-empty, we deduce that $\mathcal{F}$ is not closed, in particular it is not a circle.

Let $\gamma\colon (-a,b)\to X$, $a,b\in (0,\infty]$, be a locally minimising, unit speed curve for which $\gamma_0=x$, $x$ being as in the assumption \eqref{liibalaiba}, and $\gamma(-a,b)$ is the maximal connected subset of $\mathcal{F}$ containing $x$, and let $\varepsilon>0$ be such that $\gamma\lvert_{(-\varepsilon,\varepsilon)}$ is an isometry. 

If $a=b=\infty$, we have that $\gamma(-a,b)=X$ by the following argument. Suppose that there exists a limit point $y\in \overline{\im\gamma}\setminus \im\gamma$ of $\gamma$. Let $(t_i)\subset\R$ be a sequence so that $\gamma_{t_i}\longrightarrow y$. We may assume that $t_i$ is increasing. If $(t_i)$ is bounded, then there exists $t_\infty$ so that $t_i\longrightarrow t_\infty$, and hence by continuity we would have that $y=\gamma_{t_\infty}$. Hence $t_i\longrightarrow \infty$. Let now $\alpha$ be a geodesic from $x$ to $y$, and let $s\coloneqq \inf\{t: \alpha_t\notin \im\gamma\}$. We know that $s>0$, since the neighbourhood of $x$ is isometric to an open interval. On the other hand, since $\alpha$ is a geodesic and $l(\gamma)=\infty$, we know that there exists a sequence $(s_i)$ so that $s_i\longrightarrow s_\infty<\infty$, and $\gamma_{s_i}\longrightarrow \alpha_s$. Hence, $\gamma_{s_\infty}=\alpha_s$. This is in contradiction with the definition of $s$, since $\gamma_{s_\infty}$ has a neighbourhood of the form $\gamma(s_\infty-\delta,s_\infty+\delta)$. Thus, $\gamma(-a,b)$ is open and closed, and hence $\gamma(-a,b)=X$ giving a contradiction.

By the above, we may assume that $b<\infty$. Let $y\in X\setminus \im\gamma$, $\alpha$ be a unit speed geodesic from $x$ to $y$, and, as above, $s\coloneqq \inf\{t: \alpha_t\notin \im\gamma\}$. Since by the arguments above, $\gamma_t\neq \alpha_s$ for any $t\in(-a,b)$, there exists a sequence $(t_i)$ that (we may assume without loss of generality to) converge to $b$ so that $\gamma_{t_i}\longrightarrow \alpha_s$. By the maximality of $\im\gamma$, we know that $B(\alpha_s,r)\setminus (\im\gamma\cup\im \alpha)\neq\varnothing$ for any $r>0$. Let now $z\in B(\alpha_s,r)\setminus (\im\gamma\cup\im \alpha)$, where $l(\alpha)>r>0$ is chosen small enough so that any geodesic from $x$ to $z$ goes through the point $\alpha_s$. We are ready to arrive to a contradiction with \eqref{exmaps}. Let $\beta$ be a unit speed geodesic from $x$ to $z$. Take $t_1>s$ so that $d(\alpha_s,\alpha_{t_1})=d(\alpha_s,z)$, and define measures $\mu_0\coloneqq \frac1{\m(\gamma([0,\varepsilon)))}\m\lvert_{\gamma([0,\varepsilon))}\in\P_2^{ac}(X)$, and $\mu_1\coloneqq \frac12((\alpha\circ\restr_{t_0}^{t_1})_\# \mathcal{L}^1+(\beta\circ\restr_{t_0}^{t_1})_\#\mathcal{L}^1)\in\P_2(X)$, where $t_0$ is chosen so that $\alpha_t\neq \beta_t$ for every $t\in[t_0,t_1]$. 

Then, by using $c$-cyclical monotonicity as follows, we deduce that any plan from $\mu_0$ to $\mu_1$ cannot be given by a map, giving a contradiction with the assumption \eqref{exmaps}. Indeed, if $\sigma\in\opt{\mu_0}{\mu_1}$ is induced by a map, there exists a $c$-cyclically monotone set $\Gamma\subset X\times X$ so that $\sigma(\Gamma)=1$, and $\Gamma_{\tilde x}\coloneqq \{y\in X: (\tilde x,y)\in \Gamma\}$ is a singleton for all $\tilde x\in \mathrm{P}^1(\Gamma)$. By the definition of $\mu_1$, and the fact that $\sigma(\Gamma)=1$, we have that there exist $(x_1,\alpha(t)), (x_2,\beta(t))\in \Gamma$ with some $t\in [t_0,t_1]$. Since $\Gamma_{\tilde x}$ is a singleton for all $\tilde x$, we deduce that $x_1\neq x_2$. Then by the definitions of $\mu_0$ and $\mu_1$, there exists $\tilde x$ between $x_1$ and $x_2$ so that $(\tilde x,y)\in\Gamma$ for some $y\in X$, and so that $\beta(t)\neq y\neq \alpha(t)$. Suppose that $y=\alpha(\tilde t)$, with $\tilde t<t$ (the other cases are analogous). Then
\begin{align}d^2(x_1,\alpha(t))+d^2(\tilde x,y)&=[d(\tilde x,\alpha(t))+d(\tilde x,x_1)]^2+[d(x_1,y)-d(\tilde x,x_1)]^2
\\ &>d^2(\tilde x,\alpha(t))+d^2(x_1,y), \end{align}
which contradicts the $c$-cyclical monotonicity of the set $\Gamma$. \end{proof}

\begin{remark}
The assumption that there exists an optimal map for \emph{all} $\mu_1$, instead of only the absolutely continuous ones, is crucial in Theorem \ref{mild1d}. This can be seen by taking three non-atomic, mutually singular probability measures with full supports on the interval $[0,1]$, and pushing them to different branches of a tripod (see \cite[Example in Section 3] {Kell} for more details). The metric measure space obtained in this way satisfies the assumption of existence of transport maps between absolutely continuous measures, but does not satisfy the conclusion of Theorem \ref{mild1d}.
\end{remark}

One might wonder, whether after relaxing the assumption \eqref{exmaps} to concern only absolutely continuous measures $\mu_1$ one could still conclude that the space is one-dimensional in some appropriate sense (as in the case of the above-mentioned tripod). The following example shows that this is not the case in general.

\begin{example} \label{1d2d} 
We will construct a metric measure space having -- as a metric space -- a one-dimensional part (a line segment), and a two-dimensional part (a circular sector) which has the property that the optimal plan between any two absolutely continuous Borel probability measures is unique and induced by a map (see Figure \ref{viuhka}).  
\begin{figure}[h]
\begin{center}
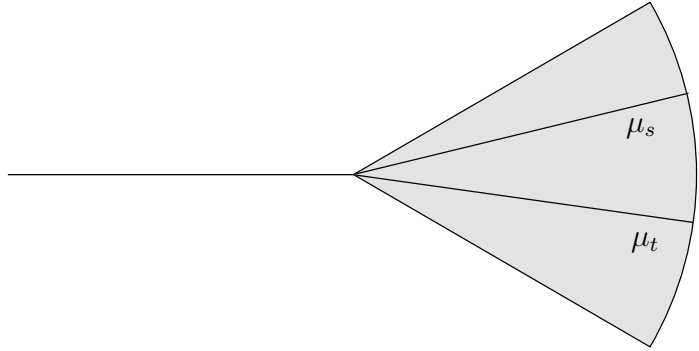
\caption{The space satisfying the weakened assumptions of Theorem \ref{mild1d}, but failing the conclusion.}
\label{viuhka}
\end{center}
\end{figure}
Let us first define auxiliary measures on the unit square $I\times I$. Let $f\colon 2^\N\to I$ be a map defined as 
\[(x_i)_{i\in\N}\mapsto \liminf_{N\to\infty}\frac1N\sum_{i=1}^N x_i.\]
Define a family of measures $\{\mu_t\}_{t\in I}$ on $I\times I$ as a pushforward of the measures 
\[\nu_t\coloneqq\tens_\N [(1-t)\delta_0+t\delta_1]\]
under the graph map $(\iota,f)$, where $\iota\colon 2^\N\to I$ is the Borel isomorphism (up to removing a countable set) $(x_i)_{i\in\N}\mapsto \sum_i x_i2^{-i}$.

Finally, define a measure $\tilde\m$ on $I\times I$ by setting
\[\int g\d\tilde\m\coloneqq \int_I \int_{I\times I} g\d\mu_t\d \mathcal{L}^1(t)=\int_I \int_{I\times \{t\}} g\d\mu_t\d \mathcal{L}^1(t) \]
for every positive Borel function $g$. Here the second equality is due to the fact that $\nu_t(f^{-1}(t))=1$ by the strong law of large numbers. To see that the definition makes sense, we need to show that the map $t\mapsto \mu_t(A)$ is Borel for every $A\in \B(I\times I)$. It suffices to prove the Borel measurability of the map $t\mapsto \nu_t(A)$ in the case where $A$ is of the form \[A=\{x_1\}\times\cdots\times \{x_n\}\times 2^\N.\]
Indeed, by setting 
\[P\coloneqq\{U\subset 2^\N: U=\{x_1\}\times\cdots\times \{x_n\}\times 2^\N \mathrm{\ for\ some\ } x_1,\dots, x_n\in\{0,1\}\}\] and 
\[D\coloneqq \{V\subset 2^\N: t\mapsto \nu_t(V) \mathrm{\ is\ Borel}\},\] we deduce from Dynkin's $\pi-\lambda$ theorem that $\B(2^\N)\subset\sigma(P)\subset D$, if $P\subset D$. Thus, we obtain that $t\mapsto \mu_t(B\times C)$ is Borel measurable for $B,C\in\B(I)$, if $P\subset D$. Applying the $\pi-\lambda$ theorem again, we deduce that $t\mapsto \mu_t(A)$ is Borel, if $P\subset D$. 

Let now $A=\{x_1\}\times\cdots\times \{x_n\}\times 2^\N$. Then the map $t\mapsto \nu_t(A)=\Pi_{i=1}^n [(1-t)\delta_0(x_i)+t\delta_1(x_i)]$ is continuous, thus Borel. Hence, the measure $\tilde\m$ is a well-defined Borel probability measure on $I\times I$.

With $\tilde\m$ defined, we may define a metric measure space $(X,d,\m)$ in the following way. Let 
\[X=([-1,0]\times \{0\})\cup T(I\times I)\subset \R^2,\]
where $T$ is the map $(t,\theta)\mapsto te^{i(\theta-\frac12)}\in\R^2$. Define the metric $d$ on $X$ as the length metric of $X$ as a subset of $\R^2$, and the measure $\m$ as the pushforward of $\tilde\m$ under the map $T$ on $T(I\times I)$, and as the Lebesgue measure on the interval $[-1,0]\times \{0\}$. Then $\m$ is a well-defined, non-trivial and finite measure on $X$ with $\spt\m=X$.  

We will show that the metric measure space $(X,d,\m)$ satisfies the condition that for all $\mu_0, \mu_1\in\P_2^{ac}(X)$ there exists a unique plan $\pi\in\og(\mu_0,\mu_1)$, and it is induced by a map. To do so, we will use the fact, that in $\R^n$ it suffices to check that the starting measure gives zero measure to Lipschitz graph to deduce the existence and uniqueness of the optimal transport map \cite{Gangbo-McCann1996}. First of all, if $G\subset I\times I$ is a Lipschitz graph, then the intersection of $G$ with $ \mathcal{L}^1$-almost every line parallel to $x$-axis has only finitely many points (e.g.\ \cite[Theorem 10.10]{Mattila}). Thus $\tilde\m(G)=0$ due to the fact that each $\mu_t$ is non-atomic. Furthermore, since $T$ is biLipschitz on every set of the form $[\varepsilon, 1]\times I$, $\varepsilon>0$, we have that $\m\lvert_{T(I\times I)}(G)=0$ for every Lipschitz graph $G\subset X$. 

Let now $\mu_0, \mu_1\in\P_2^{ac}(X)$, and $\pi\in\og(\mu_0,\mu_1)$. Write $\pi=\pi^1+\pi^2$, where $\pi^1\coloneqq \pi\lvert_{e_0^{-1}([-1,0]\times \{0\})}$ and $\pi^2\coloneqq \pi\lvert_{e_0^{-1}(T(I\times I))}$. Since $\mu_0\lvert_{[-1,0]\times \{0\}}\ll \mathcal{L}^1$, we have by the following argument that $\pi^1$ is induced by a map. Due to the fact that the measures $\{\nu_t\}$ are concentrated on the pairwise disjoint sets $f^{-1}(t)$, we have that the measure $\tilde\m$ is concentrated on the graph of the Borel function $f\circ \iota^{-1}$. Thus, the measure $\m\lvert_{T(I\times I)}$ is concentrated on a set $F=T(\mathrm{Graph}(f\circ \iota^{-1}))$ having the property that if $te^{i\theta_1},te^{i\theta_2}\in F$, then $\theta_1=\theta_2$. Then by absolute continuity of $\mu_1$, the measure $(e_1)_\#\pi^1$ is concentrated on $F$. Moreover, by the definition of the metric $d$ we have that for points $x=(s,0)\in [-1,0]\times \{0\}$ and $y=te^{i\theta}\in T(I\times I)$, the distance is $d(x,y)=\lvert t-s\rvert$. Hence, $\sigma\coloneqq G_\#(e_0,e_1)_\#\pi^1\in\P(\R\times \R)$ is an optimal plan, where $G((s,0),(t,0))=(s,t)$ for $(t,0)\in [-1,0]\times \{0\}$, and $G((s,0),te^{i\theta})=(s,t)$ for $te^{i\theta}\in F$. By the absolute continuity of $\mu_0\lvert_{[-1,0]\times\{0\}}$ and by the definition of $G$ we have that $P^1_\#\sigma$ is absolutely continuous with respect to the Lebesgue measure. Therefore, $\sigma$ is induced by a map $S\colon \R\to\R\times\R$. Thus, $(e_0,e_1)_\#\pi^1$ is induced by the map $(x,0)\mapsto G^{-1}\circ S(x)$, and so by the uniqueness of geodesics in $X$, the plan $\pi^1$ is also induced by a map.

Thus, it remains to show that $\pi^2$ is induced by a map. Since geodesics starting from $T(I\times I)$ do not branch (i.e. $\gamma^1\lvert_{[0,t]}=\gamma^2\lvert_{[0,t]}$ implies $\gamma^1=\gamma^2$), we may assume after contracting the plan towards the starting point, that $\pi^2(\geo(T(I\times I)))=1$. Hence, by the fact that $\m\lvert_{T(I\times I)}$ gives zero measure to Lipschitz graphs, we conclude that $\pi^2$, and therefore also $\pi$, is induced by a map.
\end{example}

In spite of Example \ref{1d2d}, one may try to find sufficient conditions for the metric measure space $(X,d,\m)$ so that Theorem \ref{mild1d} would still hold after weakening the assumptions to merely consider measures $\mu_1$ that are absolutely continuous with respect to the reference measure $\m$. Naively mimicking the proof of Theorem \ref{mild1d} one gets a sufficient condition for the reference measure $\m$, namely the conclusion of Lemma \ref{polar}, leading to the following theorem. 

\begin{theorem}\label{MCPmild1d}Let $(X,d,\m)$ be an $MCP(K,N)$-space with the following properties:
\begin{enumerate}
\item For every $\mu_0,\mu_1\in\P_2^{ac}(X)$, there exists $\pi\in\og(\mu_0,\mu_1)$ that is induced by a map from $\mu_0$.
\item There exists a point $x\in X$, and a neighbourhood $B(x,r)$ isometric to an open interval in $\R$. 
\end{enumerate}
Then $X$ is a one-dimensional manifold, possibly with boundary. 
\end{theorem}
In the proof of Theorem \ref{MCPmild1d}, we will use the following lemma.
\begin{lemma}\label{polar}Let $(X,d,\m)$ be an $MCP(K,N)$-space. Then 
\begin{align}\label{polardec}
\int_X f \d\m=\int_{(0,\infty)}\int_{\partial B(x_0,r)}f\d\m_r \d\mathcal{L}^1(r),
\end{align}
for every non-negative Borel function $f$, where $\m_r$ is a finite Borel measure on $\partial B(x_0,r)$. Moreover, for any three points $x_0,y,z\in X$ with $d(x_0,y)=d(x_0,z)$, and for any $r>0$ for which $B(y,r)\cap B(z,r)=\varnothing$, there exists a set $E\subset (0,\infty)$ of positive Lebesgue measure such that $\m_r(B(y,r)),\m_r(B(z,r))>0$ for all $r\in E$. 
\end{lemma}

The identity \eqref{polardec} of Lemma \ref{polar} follows directly from Lemma \ref{BGtopolar} and the Bishop-Gromov theorem in $MCP(K,N)$-spaces proven by Ohta in \cite{Ohta2}. The rest of the proof of Lemma \ref{polar} goes exactly as the proof of \cite[Lemma 2.13]{Kitabeppu-Lakzian}, after proving Lemma \ref{claim216} (see \cite[Claim 2.16]{Kitabeppu-Lakzian}).

\begin{lemma}\label{BGtopolar}
Let $(X,d,\m)$ be a metric measure space satisfying the following local Bishop--Gromov volume comparison property: for every $x_0\in X$ and $R>0$ there exists a locally absolutely continuous $w\coloneqq w_{x_0,R}\colon (0,R)\to (0,\infty)$ such that the map
\begin{align}\label{Bishop-Gromov}r\mapsto  \frac{\m(B(x,r))}{w(r)}, \quad B(x,r)\subset B(x_0,R)
\end{align}
is decreasing. Then we can write $\m$ as
\begin{align}\int_X f \d\m=\int_{(0,\infty)}\int_{\partial B(x_0,r)}f\d\m_r \d\mathcal{L}^1(r),
\end{align}
for every non-negative Borel function $f$.
\end{lemma}

\begin{proof}First of all, notice that by the Bishop-Gromov inequality \eqref{Bishop-Gromov} we have that the measure is boundedly finite (since it is locally finite by assumption). Thus, by the disintegration theorem we have that
\begin{align}\int_X f \d\m=\int_{(0,\infty)}\int_{\partial B(x_0,r)}f\d\m_r \d\varpi(r),
\end{align}
where $\varpi$ is the pushforward of $\m$ with respect to the function $x\mapsto d(x_0,x)$, and $\m_r$ is a finite Borel measure. We need to show that $\varpi$ is absolutely continuous with respect to Lebesgue measure. Observe first, that for $r<r_1<r_2<R$ we have that
\begin{align}\varpi([r_1,r_2])&= \m( B(x,r_2))-\m(B(x,r_1))
\\ &\le (w(r_2)-w(r_1))\frac{\m(B(x,r_1))}{w(r_1)}
\\ &\le C_{r,R}(w(r_2)-w(r_1)).
\end{align}
Thus, by the absolute continuity we have that for any $\varepsilon>0$ there exists $\delta>0$ so that 
\begin{align}\varpi\left(\bigsqcup_{i\in \N} [r_i,s_{i}]\right)\le C_{r,R} \sum_{i\in\N} (w_{s_i}-w_{r_i})<\varepsilon, 
\end{align}
whenever $\mathcal{L}^1\left(\bigsqcup_{i\in \N} [r_i,s_{i}]\right)<\delta$. Hence, by the definition of Lebesgue measure, we have that $\varpi \ll \mathcal{L}^1$.
\end{proof}

\begin{definition}[Non-degenerate measure \cite{Cavalletti-Mondino2016}] A measure $\m$ is called \emph{non-degenerate}, if for every Borel subset $A$ with $\m(A)>0$ we have that $\m(A_{t,x})>0$ for every $x\in X$ and $t\in[0,1)$, where $A_{t,x}\coloneqq \{\gamma_t: \gamma\in\geo(X), \gamma_0\in A, \gamma_1=x\}$. 
\end{definition}
\begin{lemma}\label{claim216}Let $(X,d,\m)$ be a metric measure space with $\m$ non-degenerate and satisfying the Bishop-Gromov inequality \eqref{Bishop-Gromov}. Let $x,y\in X$ with $l\coloneqq d(x,y)$, let $0<r_0<l$ and let $E\coloneqq \{r\in (l-r_0,l): \m_r(B(y,r_0))=0\}$. Then for every $t\in(0,1]$, we have that
\begin{align}\mathcal{L}^1\Big(\Big(\big(\frac1tE\big)\cap\big(l-r_0,l\big)\Big)\setminus E\Big)=0.
\end{align}
\end{lemma}
\begin{proof} Suppose the contrary. Then there exists $t$ so that 
\[\mathcal{L}^1(F)>0,\]
where $F\coloneqq((\frac1tE)\cap(l-r_0,l))\setminus E$. Define $B_F\coloneqq \{z\in B(y,r_0): d(x,z)\in F\}$. Then we have that
\[\m(B_F)=\int_F\m_r(B(y,r_0))\d\mathcal{L}^1(r)>0.\]

By the non-degeneracy assumption we have that $\m((B_F)_{t,x})>0$. Therefore, we have that 
\[\mathcal{L}^1(tF\setminus E)>0.\]
On the other hand, by the definition of the set $F$, we have that $tF\subset E$, which is a contradiction.
\end{proof}

With Lemma \ref{polar} at our disposal, we are ready to prove Theorem \ref{MCPmild1d}.
\begin{proof}[Proof of Theorem \ref{MCPmild1d}]
Suppose that the claim is not true. As in the proof of Theorem \ref{mild1d}, we find a pair of branching geodesics $\alpha$ and $\beta$ with equal length, which start from $x$ and end up to two distinct points. Let $r_0$ be such that $B(\alpha(1),r_0)\cap B(\beta(1),r_0)=\varnothing$. Let $E$ be the set given by Lemma \ref{polar}, and $A\coloneqq \{z\in B(\alpha(1),r_0)\cup B(\beta(1),r_0): d(x,z)\in E\}$. Let $\varepsilon>0$ be such that $B(x,\varepsilon)$ is isometric to an interval. Define $\mu_0\coloneqq \frac1{\m(\gamma([0,\varepsilon)))}\m\lvert_{\gamma([0,\varepsilon))}$ and $\mu_1\coloneqq \frac1{\m(A)}\m\lvert_A$. Now, due to the definition of the set $E$, we conclude exactly with the same arguments as in the proof of Theorem \ref{mild1d} that none of the optimal plans from $\mu_0$ to $\mu_1$ is given by a map, which is a contradiction with the assumption.
\end{proof}

In Theorem \ref{mild1d} and Theorem \ref{MCPmild1d} we gained a global one-dimensionality from a local one-dimensionality near a single point in the space. In the next theorem we will weaken the assumption on the local one-dimensionality to assumption on the infinitesimal one-dimensionality at a single point. However, at the same time we need to assume not only the existence of optimal maps but also the uniqueness of the plan. Notice that the existence of optimal maps is quite often proven in such a way that the uniqueness is achieved at the same time.
\begin{theorem}\label{micro1dgen}Let $(X,d,\m)$ be a locally metrically doubling metric measure space with the following properties:
\begin{enumerate}
\item For every $\mu_0\in\P_2^{ac}(X)$ and $\mu_1\in\P_2(X)$, there exists a unique $\pi\in\og(\mu_0,\mu_1)$ and it is induced by a map from $\mu_0$.
\item There exists a point $x\in X$ such that $\mathrm{Tan}(X,x)=\{(\R,0)\}$.
\end{enumerate}
Then $X$ is a one-dimensional manifold, possibly with boundary.
\end{theorem}

In particular, we get the following stronger version of Corollary $\ref{cor1d}$ in the case of essentially non-branching spaces.
\begin{corollary}\label{microcor}
Let $(X,d,\m)$ be an essentially non-branching $MCP(K,N)$-space (or ess. nb., qualitatively non-degenerate space). Suppose that there exists a point $x\in X$ for which $\mathrm{Tan}(X,x)=\{(\R,0)\}$. Then $X$ is a one-dimensional manifold, possibly with boundary.
\end{corollary}
It is worth mentioning that the case of branching very strict $CD(K,N)$ -spaces  is more difficult due to the non-uniqueness of the optimal transport map, and that it is not clear how to overcome this issue to obtain the same result in that case.

\begin{proof}[Proof of Theorem \ref{micro1dgen}]
Let $x\in X$ be such that $\Tan(X,x)=\{(\R,0)\}$. By Theorem \ref{mild1d} it suffices to prove that $x$ has a neighbourhood isometric to an open interval. 
Let $\varepsilon_n\searrow 0$. Then by the fact that $\R$ is a tangent of $X$ at $x$, we find $\tilde r_n\searrow 0$ and points $y_n,z_n\in X$ with 
$d(y_n,z_n)\ge2\tilde r_n-\tilde r_n\varepsilon_n$, and $\tilde r_n+\tilde r_n\varepsilon_n\ge d(y_n,x),d(z_n,x)\ge \tilde r_n-\tilde r_n\varepsilon_n$. Let $\delta_n\coloneqq d(y_n,x)$. We may assume that $B(y_n,\frac{\delta_n}2)$ is not isometric to an open interval (otherwise we are done by Theorem \ref{mild1d}).

Suppose that $x$ has no neighbourhood isometric to an open interval. Hence, by the following argument, there exist two geodesics $\alpha^n$ and $\beta^n$ from $x$ to $B(y_n,\frac{\delta_n}2)$ with $l(\alpha^n)=l(\beta^n)$ and $\alpha^n_1\neq\beta^n_1$. Let $y'_n\in B(y_n,\delta_n)$ be a point lying in the interior of a geodesic $\gamma$ connecting $x$ to $y_n$. Let $B$ be a ball centered at $y'_n$ so that $B\subset B(y_n,\frac{\delta_n}2)\setminus\{y_n\}$. The ball $B$ is not isometric to an interval (again by Theorem \ref{mild1d}). Thus, it contains a point $y''_n$ which does not lie in the geodesic $\gamma$. Then, by triangle inequality, there exists a point $w_n$ lying in $\gamma$ with $d(x,y_n'')=d(x,w_n)$. The geodesics $\alpha^n$ and $\beta^n$ connecting $x$ to $y''_n$ and $w_n$, respectively, satisfy the desired conditions.

Define measures $\mu^n_0\coloneqq \frac1{\m(B(z_n,\frac{\delta_n}2))}\m\lvert_{B(z_n,\frac{\delta_n}2)}$ and $\mu^n_1\coloneqq \frac12[\delta_{\alpha^n_1}+\delta_{\beta^n_1}]$. Let $\pi_n\in\og(\mu^n_0,\mu_1^n)$ be the unique optimal plan. Then there exists a $\pi_n$- positive measure set $\mathcal A_n$ such that any $\eta\in\mathcal A_n$ does not go through $x$, since otherwise the uniqueness of the plan could be easily violated by simply sending any point to both of the points $\alpha^n_1$ and $\beta^n_1$. More precisely, suppose that $\pi_n$- almost every $\gamma$ goes through $x$. This implies that for $\mu_0^n$- almost every $\tilde x$ we have that 
\begin{align}\label{TAMANAIN}d(\tilde x,\alpha^n_1)=d(\tilde x,x)+d(x,\alpha^n_1)=d(\tilde x,x)+d(x,\beta^n_1)= d(\tilde x,\beta^n_1).\end{align}
Let $M\subset X$ be the set of full $\mu_0$-measure whose points satisfy $\eqref{TAMANAIN}$. Then $M\times \{\alpha^n_1,\beta^n_1\}$ is $c$-cyclically monotone set. Therefore, the measure $\mu_0^n\tens \mu_1^n$ is concentrated on a $c$-cyclically monotone set, and thus $\mu_0^n\tens \mu_1^n\in\opt{\mu^n_0}{\mu^n_1}$. Clearly, $\mu_0^n\tens \mu_1^n$ is not induced by a map, but $(e_0,e_1)_\#\pi_n$ is, which contradicts the uniqueness of the optimal plan.

Let $\eta^n\in \mathcal A_n$ be a geodesic not going via $x$. Consider the sequence $(X,d_{\delta_n},x)$, $d_{\delta_n}\coloneqq \frac{d}{\delta_n}$, which converges (up to subsequence) to $\R$. Since $d(\eta^n_t,x)\le 8\delta_n$ for all $t\in[0,1]$, we know that $\eta^n$ converges (again, up to subsequence) to a geodesic $\eta^\infty\in\geo(\R)$. Morever, since $d(y_n,\eta^n_0), d(z_n,\eta^n_1)< \frac12 \delta_n $, and since $z_n\longrightarrow -1$ and $y_n\longrightarrow1$ (or vice versa), we have that there exists a sequence $(s_n)$ so that  $\eta^n_{s_n}\longrightarrow 0$. In particular, if we take $s_n$ so that $d(\eta^n_{s_n},x)=d(\im\eta^n,x)$, we have that $\eta^n_{s_n}\longrightarrow 0$ implying that $\frac{d(\eta^n_{s_n},x)}{\delta_n}\longrightarrow 0$. Thus, we find $s^1_n<s^2_n$ such that $d(\eta^n_{s^1_n},\eta^n_{s_n})=d(x,\eta^n_{s_n})=d(\eta^n_{s^2_n},\eta^n_{s_n})\eqqcolon r_n$ and $d(\eta^n_{s^2_n},\eta^n_{s^1_n})=2d(x,\eta^n_{s_n})$ (see Figure \ref{RR}).

\begin{figure}[h] 
\begin{center}

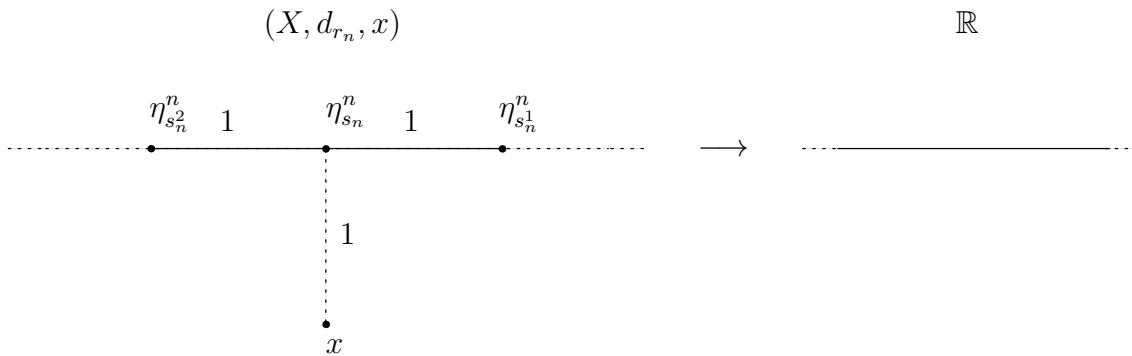
\caption{Tripod type configuration leading to a contradiction at the blow-up.}
\label{RR}
\end{center}
\end{figure}

Finally, consider a sequence $(X,d_{r_n},x)$, $d_{r_n}\coloneqq \frac{d}{r_n}$, converging by the local doubling property (up to subsequence) to $\R$. Then we have again that $\eta^n$ (sub)converges to a limit geodesic in $\R$. Now by the choice of $r_n$ we have that $\eta^n_{s_n}$ converges either to $1$ or to $-1$. We may assume without loss of generality that it converges to $1$. Then, by the choice of $s^1_n$ and $s^2_n$, we have that $\restr_{s^1_n}^{s_n^2}(\eta^n)$ converges to an interval $[0,2]$, implying that $\eta^n_{s^i_n}$ converges to $0$ for either $i\in\{1,2\}$. This contradicts the fact that $d(\eta^n_{s_n^i},x)\ge d(\eta^n_{s_n},x)=r_n$ for both $i\in\{1,2\}$.

\end{proof}
The following example shows that the uniqueness of the optimal map is crucial in Theorem \ref{micro1dgen}.
\begin{example}[Necessity of the uniqueness of the map in Theorem \ref{micro1dgen}]
Let $X\subset \R^2$ be a weakly (geodesically) convex, two-sided cusp on the plane endowed with supremum norm, for example $A\coloneqq \{(x,y)\in\R^2: \lvert x\rvert \le\frac12,\lvert y\rvert\le x^2\}$. Let $d$ be the distance induced by the supremum norm restricted to $X$, and $\m$ a locally finite measure absolutely continuous with respect to $\mathcal{L}^2$. Then $(X,d,\m)$ is a geodesic (since $\lvert (x^2)'\rvert\le1$) metric measure space satisfying 
\begin{enumerate}
\item For every $\mu_0\in\P_2^{ac}(X)$ and $\mu_1\in\P_2(X)$, there exists $\pi\in\og(\mu_0,\mu_1)$ that is induced by a map from $\mu_0$\footnote{The existence of such $\pi$ is due to the existence of optimal maps on $(\R^2,\lvert\cdot \rvert_{\sup})$, see for example \cite{Carlier-DePascale-Santambrogio}, and the weak convexity of $X$. },
\item$\mathrm{Tan}(X,0)=\{(\R,0)\}$,
\end{enumerate}
but does not satisfy the conclusion of Theorem \ref{micro1dgen}.
\end{example}

\bibliographystyle{amsplain}
\bibliography{One_dim_ref}

\end{document}

%% file: Viuhka2.pdf_t
\pdfoutput=1
\begin{picture}(0,0)%
\includegraphics{Viuhka2.pdf}%
\end{picture}%
\setlength{\unitlength}{4144sp}%
\begingroup\makeatletter\ifx\SetFigFont\undefined%
\gdef\SetFigFont#1#2#3#4#5{%
  \reset@font\fontsize{#1}{#2pt}%
  \fontfamily{#3}\fontseries{#4}\fontshape{#5}%
  \selectfont}%
\fi\endgroup%
\begin{picture}(4141,2079)(1204,-4112)
\put(4914,-2815){\makebox(0,0)[lb]{\smash{{\SetFigFont{12}{14.4}{\familydefault}{\mddefault}{\updefault}{\color[rgb]{0,0,0}$\mu_s$}%
}}}}
\put(4944,-3502){\makebox(0,0)[lb]{\smash{{\SetFigFont{12}{14.4}{\familydefault}{\mddefault}{\updefault}{\color[rgb]{0,0,0}$\mu_t$}%
}}}}
\end{picture}%

%% file: Ristiriita.pdf_t
\pdfoutput=1
\begin{picture}(0,0)%
\includegraphics{Ristiriita.pdf}%
\end{picture}%
\setlength{\unitlength}{4144sp}%
\begingroup\makeatletter\ifx\SetFigFont\undefined%
\gdef\SetFigFont#1#2#3#4#5{%
  \reset@font\fontsize{#1}{#2pt}%
  \fontfamily{#3}\fontseries{#4}\fontshape{#5}%
  \selectfont}%
\fi\endgroup%
\begin{picture}(6783,2137)(439,-4019)
\put(2352,-3955){\makebox(0,0)[lb]{\smash{{\SetFigFont{12}{14.4}{\familydefault}{\mddefault}{\updefault}{\color[rgb]{0,0,0}$x$}%
}}}}
\put(1296,-2540){\makebox(0,0)[lb]{\smash{{\SetFigFont{12}{14.4}{\familydefault}{\mddefault}{\updefault}{\color[rgb]{0,0,0}$\eta^n_{s_n^2}$}%
}}}}
\put(2352,-2540){\makebox(0,0)[lb]{\smash{{\SetFigFont{12}{14.4}{\familydefault}{\mddefault}{\updefault}{\color[rgb]{0,0,0}$\eta^n_{s_n}$}%
}}}}
\put(3387,-2540){\makebox(0,0)[lb]{\smash{{\SetFigFont{12}{14.4}{\familydefault}{\mddefault}{\updefault}{\color[rgb]{0,0,0}$\eta^n_{s^1_n}$}%
}}}}
\put(2436,-3300){\makebox(0,0)[lb]{\smash{{\SetFigFont{12}{14.4}{\familydefault}{\mddefault}{\updefault}{\color[rgb]{0,0,0}1}%
}}}}
\put(2817,-2624){\makebox(0,0)[lb]{\smash{{\SetFigFont{12}{14.4}{\familydefault}{\mddefault}{\updefault}{\color[rgb]{0,0,0}1}%
}}}}
\put(1718,-2624){\makebox(0,0)[lb]{\smash{{\SetFigFont{12}{14.4}{\familydefault}{\mddefault}{\updefault}{\color[rgb]{0,0,0}1}%
}}}}
\put(6121,-2041){\makebox(0,0)[lb]{\smash{{\SetFigFont{12}{14.4}{\familydefault}{\mddefault}{\updefault}{\color[rgb]{0,0,0}$\R$}%
}}}}
\put(4576,-2776){\makebox(0,0)[lb]{\smash{{\SetFigFont{12}{14.4}{\rmdefault}{\mddefault}{\updefault}{\color[rgb]{0,0,0}$\longrightarrow$}%
}}}}
\put(1981,-2041){\makebox(0,0)[lb]{\smash{{\SetFigFont{12}{14.4}{\familydefault}{\mddefault}{\updefault}{\color[rgb]{0,0,0}$(X,d_{r_n},x)$}%
}}}}
\end{picture}%